\documentclass[11pt]{article}
\usepackage{a4wide}
\usepackage{amsmath, amsthm, xcolor}
\usepackage{amssymb}
\usepackage{graphicx,soul}
\usepackage[normalem]{ulem}
\usepackage{relsize}
\usepackage{enumerate}
\newcommand{\hidden}[1]{}
\usepackage[colorlinks, linkcolor=blue, anchorcolor=blue, citecolor=blue]{hyperref}
\usepackage{bm}
\usepackage{color}
\usepackage[mathscr]{euscript}
\usepackage{tikz}
\numberwithin{equation}{section}

\newtheorem{theorem}{Theorem}[section]
\newtheorem{proposition}[theorem]{Proposition}

\newtheorem{corollary}[theorem]{Corollary}
\newtheorem{lemma}[theorem]{Lemma}

\newtheorem{definition}{Definition}
\theoremstyle{remark}
\newtheorem{remark}[theorem]{{Remark}}

\begin{document}

\title{$P$-adic $s$-energy and Fourier dimension}


 \author{Cheng Liu \\ {\small\sc (Wuhan) } \and Lingmin Liao \\ {\small\sc (Wuhan) }}

\bigskip

\date{}

 \maketitle

\begin{abstract}
	We extend the theory of $s$-energy to $\mathbb{Q}_p^n$. We give an exact representation for the $s$-energy in terms of the Fourier coefficients of a finite Borel measure with support in $\mathbb{Z}_p^n$, which is stronger than the approximate representation in $\mathbb{R}^n$. We also compute the Fourier dimension of the $p$-adic Riesz product measure, even in the case $p=2$, which has never been studied before.
\end{abstract}
$\textbf{Keywords}$: $p$-adic harmonic analysis,
Fourier dimension, $p$-adic Riesz product
\bigskip

\section{Introduction}
Let $G$ be a locally compact abelian group equipped with a metric $d$. The \textit{support} of a measure $\mu$ on $G$, denoted by $\text{spt}\mu$, is the smallest closed set $F$ such that $\mu(G\setminus F)=0$. For $A\subset G$, the set of all Borel measures $\mu$ on $G$ with $0<\mu(A)<+\infty$ and with compact $\text{spt}\mu\subset A$ will be denoted by $\mathcal{M}(A)$. For any real number $s>0$, define the \textit{$s$-energy} of a measure $\mu\in \mathcal{M}(A)$ by
\begin{align}\label{isg} 
  I_s^G(\mu)=\int_G\int_G d(x,y)^{-s}d\mu(x)d\mu(y).
\end{align}
If $G=\mathbb{R}^n$, we abbreviate $I_s^G(\mu)$ as $I_s(\mu)$. The properties and applications of $I_s(\mu)$
have been extensively studied. This quantity is closely connected to a number of fundamental problems in geometric measure theory, including those involving Hausdorff dimension, Fourier dimension, the distance problem, and Marstrand's projection theorem. For further details, we refer the reader to \cite{k68,m15}.\par
Let $p\ge 2$ be a prime. Let $\mathbb{Q}_p$ be the field of $p$-adic numbers and $\mathbb{Q}_p^n$ be the $n$-th Cartesian product of $\mathbb{Q}_p$. 
Recently, $p$-adic harmonic analysis has attracted considerable attention. In 2019, Fan, Fan, Liao and Shi \cite{ffls19} resolved the $p$-adic Fuglede conjecture in dimension one. In 2024, Arsovski \cite{a24} settled the $p$-adic Kakeya conjecture. Subsequently, in 2025, Fan and Fan \cite{ff25} established the existence of a Riesz basis of exponentials for every finite union of balls in $\mathbb{Q}_p^n$, as well as the non-existence of such a basis for certain bounded sets in $\mathbb{Q}_p^n$. \par
We now introduce the main object of study of this paper. Let $G=\mathbb{Q}_p^n$ and $\mu\in\mathcal{M}(\mathbb{Q}_p^n)$. Then \eqref{isg} becomes
\[I_s^p(\mu)=\int_{\mathbb{Q}_p^n}\int_{\mathbb{Q}_p^n} |x-y|_p^{-s}d\mu(x)d\mu(y),\]
called the \textit{$p$-adic $s$-energy} of a Borel measure $\mu$.
\par
Denote the \textit{diameter} of a set $A\subset \mathbb{Q}_p^n$ by $d(A)$. For any Borel set $A\subset\mathbb{Q}_p^n$ and any real number $s\ge 0$, define \textit{Hausdorff measure} $\mathcal{H}^s$ by 
\[\mathcal{H}^s(A)=\lim_{\delta\to 0}\mathcal{H}^s_{\delta}(A),\]
where, for $0<\delta\le +\infty$
\[\mathcal{H}^s_{\delta}(A)=\inf\left\{\sum_j d(E_j)^s: A\subset \bigcup_jE_j,\ d(E_j)<\delta\right\}.\]
The \textit{Hausdorff dimension} of $A\subset\mathbb{Q}_p^n$ is
\[\dim_H(A)=\inf\left\{s:\mathcal{H}^s(A)=0\right\}=\sup\left\{s:\mathcal{H}^s(A)=+\infty\right\}.\]
The following theorem gives a relationship between Hausdorff dimension and $p$-adic $s$-energy.
\begin{theorem}\label{hwdl}
  For a Borel set $A\subset \mathbb{Q}_p^n$,
  \[\dim_H(A)=\sup\{0\le s\le n:\ \text{there exists}\ \mu\in \mathcal{M}(A),\ \text{such that}\ I_s^p(\mu)<+\infty\}.\]
\end{theorem}
Since $\mathbb{Q}_p^n$ is a locally compact abelian group, there exists a unique Haar measure $dx$ with $\int_{\mathbb{Z}_p^n}1dx=1$, where $\mathbb{Z}_p^n:=\{x\in\mathbb{Q}_p^n: |x|_p\le 1\}$ is the ring of $p$-adic integers. Let $\chi(x)=e^{2\pi i\{x\}_p}$ for all $x\in\mathbb{Q}_p$, where $\{x\}_p$ is defined as in \eqref{def of fp}. The Fourier transform of a Borel measure $\mu$ on $\mathbb{Q}_p^n$ is defined by 
\[\widehat{\mu}(\xi)=\int_{\mathbb{Q}_p^n} \chi(-\xi x)d\mu(x),\ \forall \xi\in\mathbb{Q}_p^n.\]
The following formula is a key to link the Hausdorff dimension to the Fourier transform.
\begin{theorem}\label{thm1.2}
  Let $\mu\in\mathcal{M}(\mathbb{Q}_p^n)$ and $0<s<n$. Then,
  \[I_s^p(\mu)=\frac{1-p^{-s}}{1-p^{s-n}}\int_{\mathbb{Q}_p^n} |\widehat{\mu}(x)|^2|x|_p^{s-n}dx.\]
\end{theorem}
Also, in $\mathbb{Q}_p^n$ we have a better conclusion which accurately provide the value of $I_s^p(\mu)$ rather than giving approximate estimates.
\begin{theorem}\label{qiuhe}
  Let $\mu\in\mathcal{M}(\mathbb{Z}_p^n)$ and $0<s<n$. Then,
  \[I_s^p(\mu)=\frac{1-p^{-n}}{1-p^{s-n}}\mu(\mathbb{Z}_p^n)^2+\frac{1-p^{-s}}{1-p^{s-n}}\underset{z\in\widehat{\mathbb{Z}}_p^n\setminus \{0\}}{\sum}|\widehat{\mu}(z)|^2|z|_p^{s-n}.\]
\end{theorem}
We now introduce the definitions of the Fourier dimension of a Borel set $A\subset \mathbb{Q}_p^n$ and of a measure $\mu\in\mathcal{M}(\mathbb{Q}_p^n)$. We write 
$f\lesssim g$ to mean that there exists a constant $C>0$, independent of $f$ and $g$, such that $f\le Cg$.
\begin{definition}\label{fws}
  The Fourier dimension of a Borel set $A\subset \mathbb{Q}_p^n$ is defined by
  \begin{align*}
    \dim_F(A)=\sup\{0\le s\le n:\ \exists\ \mu\in\mathcal{M}(A),\ \text{s.t.}\ |\widehat{\mu}(x)|\lesssim |x|_p^{-\frac{s}{2}},\ \forall\ x\in\mathbb{Q}_p^n\}.
  \end{align*}
  The Fourier dimension of a measure $\mu\in\mathcal{M}(\mathbb{Q}_p^n)$ is defined by
  \begin{align*}
    \dim_F(\mu)=\sup\{s\ge 0:|\widehat{\mu}(x)|\lesssim |x|_p^{-\frac{s}{2}}\ \text{for all}\ x\in\mathbb{Q}_p^n\}.
  \end{align*}
\end{definition}
\begin{remark}
  It should be noted that these definitions differ slightly from the classical ones in $\mathbb{R}^n$, since we need to impose the additional restriction $\dim_F(A)\le n$, which holds automatically in the real setting.
\end{remark}
The following theorem shows that for any Borel set $A\subset\mathbb{Q}_p^n$, the Fourier dimension does not exceed the Hausdorff dimension. Consequently, in analogy with the real case, we define a Salem set in $\mathbb{Q}_p^n$ to be a Borel set $A$ satisfying $\dim_H(A)=\dim_F(A)$.
\begin{theorem}\label{thm1.3}
  For any Borel set $A\subset \mathbb{Q}_p^n$, we have
  \[\dim_F(A)\le \dim_H(A).\]
\end{theorem}
\begin{remark}\label{rmex1}
  Without imposing the condition $s\le n$ in the definition of $\dim_F(A)$, one may obtain $\dim_F(A)>\dim_H(A)$. For example, consider the set $\mathbb{Z}_p^n$ and the measure $\mu$ which is the Haar measure restricted in $\mathbb{Z}_p^n$, then $\dim_F(\mu)=+\infty$, since $\widehat{\mu}(x)=0$ for all $|x|_p>1$. 
\end{remark}
Now, we consider the $p$-adic Riesz products and give their Fourier dimension. Let $p\ge 2$ be a prime. Let $\{a_k\}$ be a sequence in $\mathbb{C}$ with $0<|a_k|\le 1$ and let $\{\ell_k\}$ be a sequence in $\mathbb{N}$ with $\ell_{k+1}>\ell_{k}\ge 1$ for all $k\ge 1$. We define the $p$-adic Riesz products on $\mathbb{Z}_p$ by
\[d\mu_{p,a,\ell}(x)=\prod_{k=1}^{+\infty}\big(1+\mathbf{Re}\  a_k\gamma_{\ell_k}(x)\big)dx,\]
where $\gamma_{\ell_k}(x)=e^{2\pi i\{p^{-\ell_k}x\}_p}$.
\par
In \cite{fz09}, Fan and Zhang conducted a thorough study of $p$-adic Riesz products from the perspectives of harmonic analysis and dynamical systems. The investigation of mutual absolute continuity and singularity for infinite product measures dates back to Kakutani's seminal work \cite{k48}. This topic has subsequently been explored in depth for both deterministic and random Riesz products on the circle by several authors \cite{bm74,f91,f93,ks88,p90}. For the case of random $p$-adic Riesz products, we refer the reader to \cite{sz09}. However, none of the aforementioned results on $p$-adic Riesz products treat the case $p=2$, which is considerably more intricate than the odd prime cases.
\par
We now give the Fourier dimension of $\mu_{p,a,\ell}$ for all $p\ge 2$ and any sequence $\{a_k\}$ and $\{\ell_k\}$. 
\begin{theorem}\label{thm1.5}
  Let $p\ge 2$ be a prime. Let $\{a_k\}$ be a sequence in $\mathbb{C}$ with $0<|a_k|\le 1$ and let $\{\ell_k\}$ be a sequence in $\mathbb{N}$ with $\ell_{k+1}>\ell_{k}\ge 1$ for all $k\ge 1$. Then,
  \[\dim_F(\mu_{p,a,\ell})= 2\underset{n\to +\infty}{\varliminf}\frac{\log(|a_n|)}{-\ell_n \log(p)}.\]
\end{theorem}
\begin{remark}
  In computing the Fourier dimension of $\mu_{p,a,\ell}$, two natural perspectives arise. On the one hand, one may regard $\mu_{p,a,\ell}$ as a measure supported on $\mathbb{Z}_p$, in which case it suffices to consider the Fourier transform $\widehat{\mu}_{p,a,\ell}(\xi)$ for $\xi\in\widehat{\mathbb{Z}}_p$. On the other hand, $\mu_{p,a,\ell}$ may be viewed as a measure on $\mathbb{Q}_p$ with $\operatorname{spt}\mu_{p,a,\ell}=\mathbb{Z}_p$, requiring the evaluation of $\widehat{\mu}_{p,a,\ell}(\xi)$ for all $\xi\in\mathbb{Q}_p$. These two viewpoints yield the same Fourier dimension. Consequently, throughout the proof of our theorem, we treat $\mu_{p,a,\ell}$ as a measure on $\mathbb{Z}_p$.
\end{remark}
For two measures $\mu$ and $\nu$, we write $\mu\sim\nu$ to  indicate that they are mutually absolutely continuous. As an easy corollary, we obtain a sufficient condition ensuring that two $p$-adic Riesz products are mutually absolutely continuous.
\begin{corollary}\label{jdlx}
  Let $p\ge 3$. Suppose $\mu_{p,a,\ell}$ and $\mu_{p,b,\ell}$ are two $p$-adic Riesz products with positive Fourier dimension. Then $\mu_{p,a,\ell} \thicksim \mu_{p,b,\ell}$.
\end{corollary}
We organize this paper as follows. In Section 2, we provide some preliminaries of $p$-adic analysis and $p$-adic Riesz products. In Section 3, we prove Theorems \ref{hwdl}, \ref{thm1.2} and \ref{thm1.3} which extend the theories of Fourier dimension on $\mathbb{R}^n$ to $\mathbb{Q}_p^n$. In Section 4, we prove Theorem \ref{thm1.5}.
\section{Preliminaries}
\subsection{$P$-adic numbers and integration theory}
Let $p\ge 2$ be a prime. Any nonzero rational number $r$ can be written as $r=p^v\frac{a}{b}$ where $v,a,b\in\mathbb{Z}$ and ${\rm gcd}(p,a)={\rm gcd}(p,b)=1$. Define $v_p(r)=v$ and $|r|_p=p^{-v_p(r)}$ for $r\neq 0$ and $|0|_p=0$. Then $|\cdot |_p$ is a non-Archimedean absolute value on $\mathbb{Q}$. The completion of $\mathbb{Q}$ under the absolute value $|\cdot |_p$ is the field $\mathbb{Q}_p$ of \textit{$p$-adic numbers}. It is known that, any $x\in\mathbb{Q}_p$ has a unique expansion as 
\[x=\sum_{i=v_p(x)}^{+\infty}x_ip^i\  \text{with}\ x_i\in\{0,1,\cdots, p-1\}\ \text{and}\ x_{v_p(x)}\neq 0.\]
Let $\mathbb{Q}_p^n$ be the $n$-th Cartesian product of $\mathbb{Q}_p$. The norm on $\mathbb{Q}_p^n$ is defined by 
$$|(x_1,\cdots ,x_n)|_p=\max\{|x_i|_p\}.$$ 
The ring of $p$-adic integer $\mathbb{Z}_p^n$ is defined by $$\mathbb{Z}_p^n=\{x\in\mathbb{Q}_p^n:|x|_p\le 1\}.$$
We denote by
\[B(x_0,r):=\{x\in\mathbb{Q}_p^n: |x-x_0|_p\le r\}\ \text{and}\ S(x_0,r):=\{x\in\mathbb{Q}_p^n: |x-x_0|_p= r\}\]
the ball and sphere centered at $x_0$ with radius $r$ respectively. We also denote $B(0,p^m)$ by $B_m$ and $S(0,p^m)$ by $S_m$ for convenience.
\par
Since $\mathbb{Q}_p^n$ is a locally compact commutative group with respect to addition, there exists a unique Haar measure $dx$, with
\[\int_{\mathbb{Z}_p^n}1dx=1.\]
For any complex valued function $f$ on $\mathbb{Q}_p^n$, we define the integral of $f$ by
\[\int_{\mathbb{Q}_p^n}fdx:=\lim_{m\to +\infty}\int_{B_m}fdx=\sum_{m\in\mathbb{Z}}\int_{S_m}fdx,\]
if the limits (equivalently the summation) exists. In this case, we say that $f$ is \textit{integrable on $\mathbb{Q}_p^n$}. It can be proved that Fubini's theorem, 
dominated convergence theorem and Fatou's lemma also hold in $\mathbb{Q}_p^n$. We refer to \cite{aks10, vvz94} for more details. 
\subsection{$P$-adic Fourier analysis}
For any $x\in\mathbb{Q}_p$, define the fractional part of $x$ by 
\begin{align}\label{def of fp}
  \{x\}_p=\begin{cases}
    0, & x\in \mathbb{Z}_p, \\
    \sum\limits_{i=v_p(x)}^{-1}x_ip^i, & \text{otherwise}.
  \end{cases}
\end{align}
The following function
\[\chi(x)=e^{2\pi i\{x\}_p}\]
defines a non-trivial continuous additive character of the additive group $\mathbb{Q}_p$. Notably $\chi\equiv 1$ on $\mathbb{Z}_p$ but non-constant on the subgroup $p^{-1}\mathbb{Z}_p$. The Fourier transform of a complex valued function $f\in L^1(\mathbb{Q}_p^n)$ is defined by
\[\widehat{f}(\xi)=\int f(x)\chi(-\xi x)dx,\ \forall\  \xi\in\mathbb{Q}_p^n.\]
If $f,\ \widehat{f}\in L^1(\mathbb{Q}_p^n)$, then we have the \textit{inversion formula} 
\[f(x)=\int \widehat{f}(\xi)\chi(\xi x)d\xi,\ \forall x\in \mathbb{Q}_p^n.\]
For any $\xi\in\mathbb{Q}_p^n$, the following integrals about $\chi$ is well known,
\begin{align*}
  \int_{B_k}\chi(-\xi x)dx=
  \begin{cases}
    p^{nk}, & |\xi|_p\le p^{-k},\\
    0, & |\xi|_p\ge p^{-k+1},
  \end{cases}
\end{align*}
\begin{align*}
  \int_{S_k}\chi(-\xi x)dx=
  \begin{cases}
    p^{nk}(1-p^{-n}), & |\xi|_p\le p^{-k},\\
    -p^{n(k-1)}, & |\xi|_p= p^{-k+1},\\
    0, & |\xi|_p\ge p^{-k+2}.
  \end{cases}
\end{align*}
We also give the Fourier transform of some radial functions.
\begin{lemma}\cite[p43. Example 8]{vvz94}\label{frf}
  Let $f:\mathbb{Q}_p^n\rightarrow \mathbb{C}$ be a radial function, i.e., there exists $g: [0,+\infty)\rightarrow\mathbb{C}$, such that $f(x)=g(|x|_p)$ for all $x\in\mathbb{Q}_p^n$. Suppose 
  \[\sum_{\gamma=0}^{+\infty}g(p^{-\gamma})p^{-n\gamma}<+\infty.\]
  Then
  \[\widehat{f}(\xi)=|\xi|^{-n}_p\left((1-p^{-n})\sum_{\gamma=0}^{+\infty}g(p^{-\gamma}|\xi|_p^{-1})p^{-n\gamma}-g(p|\xi|_p^{-1})\right),\ \forall \xi\neq 0.\]
\end{lemma}
\par
The \textit{convolution} $f*g$ of functions $f$ and $g$ is defined by 
\[(f*g)(x)=\int f(x-y)g(y)dy,\]
and the convolution of a function $f$ and a Borel measure $\mu$ by 
\[(f*\mu)(x)=\int f(x-y)d\mu(y),\]
whenever the integrals exist.
\par
A complex valued function $f$ defined on $\mathbb{Q}_p^n$ is called \textit{uniformly locally constant} if there exists $k\in\mathbb{Z}$, such that
\[f(x+u)=f(x),\ \text{for all}\ x\in\mathbb{Q}_p^n,\ \text{for all}\ u\in B_k.\]
In Fourier analysis on $\mathbb{R}^n$, the Schwartz class of rapidly decreasing functions plays a central and convenient role. However, in the $p$-adic setting $\mathbb{Q}_p^n$, since differentiability cannot be defined in the same way as on $\mathbb{R}^n$, a different class of test functions is required. In the theory of Bruhat-Schwartz distributions, the space of \textit{Bruhat-Schwartz test functions} consists, by definition, of uniformly locally constant functions with compact support. The following proposition illustrates, to some extent, why these functions serve as a natural substitute for the classical Schwartz functions.
\begin{proposition}\cite[Proposition 2.2]{ffls19}\label{jbcz}
  Let $f\in L^1(\mathbb{Q}_p^n)$ be a complex valued integrable function. If $f$ is uniformly locally constant, then $\widehat{f}$ has compact support.
\end{proposition}
We refer to \cite{aks10, t75, vvz94} for more details about Bruhat-Schwartz distributions. 
\subsection{$P$-adic Riesz products}
Recall that the Fourier coefficients of a measure $\mu$ on $\mathbb{Z}_p$ are defined by 
$$\widehat{\mu}(\xi)=\int_{\mathbb{Z}_p}\chi(-\xi x)d\mu(x),\ \forall \xi\in \widehat{\mathbb{Z}}_p.$$ 
If $p\ge 3$ or $p=2$ and $\ell_{i+1}-\ell_{i}\ge 2$ for all $i\ge 1$, it can be proved that 
\begin{align}\label{rpfxs}
  \widehat{\mu}_{p,a,\ell}(\xi)=
  \begin{cases}
    a_1^{(\varepsilon_1)}\cdots a_n^{(\varepsilon_n)}, & \text{if}\ \xi=\sum\limits_{i=1}^{n}\varepsilon_i p^{-\ell_i},\\
    0, & \text{otherwise},
  \end{cases}
\end{align}
where $a^{(\varepsilon)}$ equals to $1, \frac{a}{2},\ \text{or}\ \frac{\bar{a}}{2}$ according to $\varepsilon = 0, 1 \ \text{or}\ -1$. Moreover, if $\ell_{i_0+1}\ge \ell_{i_{0}}+2$ for some $i_{0}$ (we allow $\ell_{i+1}=\ell_{i}$ for $i\neq i_0$), then 
$\widehat{\mu}_{p,a,\ell}(2^{-\ell_{i_0+1}})=\frac{a_{i_0+1}}{2}$. 
The following proposition provides a sufficient condition under which two $p$-adic Riesz products are mutually absolutely continuous. 
\begin{proposition}\cite[Theorem 1.1]{wz13}\label{rpac}
  Let $p\ge 3$. Suppose there exists some $c>0$ such that $p^{\ell_{k}-\ell_{k-1}}(2-|a_k+b_k|)^2>c$ for all large $k$ with $a_k\neq b_k$. Then $\mu_{p,a,\ell}\thicksim \mu_{p,b,\ell}$ if 
  \[\sum_{k=1}^{+\infty}|a_k-b_k|^2\left(1+\frac{\cos ^2(s_k-t_k)}{\sqrt{2-|a_k+b_k|}}\right)<+\infty,\]
  where $s_k=\arg(a_k-b_k),\ t_k=\arg(a_k+b_k)$.
\end{proposition}
\section{Energy-integral and Fourier dimension}
In order to prove Theorem \ref{hwdl}, we need a Frostman lemma in $\mathbb{Q}_p^n$ as in $\mathbb{R}^n$.
\begin{lemma}[$p$-adic Frostman lemma]\label{Frostman}
  Let $0\le s \le n$. For a Borel set $A\subset \mathbb{Q}_p^n$, $\mathcal{H}^s(A)>0$ if and only if there exists a measure $\mu\in \mathcal{M}(A)$ such that 
  \begin{align}\label{frostman}
      \mu(B(x,p^m))\lesssim p^{ms},\ \text{for any}\ x\in\mathbb{Q}_p^n\ \text{and any}\ m\in\mathbb{Z}.
  \end{align}
\end{lemma}
\begin{proof}
  Since $(\mathbb{Q}_p^n,|\cdot|_p)$ is a separable metric space, then by \cite[Proposition 3.1 and Theorem F]{fh25}, we complete the proof.
\end{proof}
\begin{lemma}\label{hwsl1}
  Let $A\subset \mathbb{Q}_p^n$ be a Borel set.
  If there exists $\mu\in \mathcal{M}(A)$, such that \eqref{frostman} holds for some $0\le s\le n$, then $I_t^p(\mu)<+\infty$ for any $0<t<s$.
\end{lemma}
\begin{proof}
  Fix $0<t<s$. Since $\mu\in \mathcal{M}(A)$, then there exists some $k\in\mathbb{Z}$ such that $\text{spt}\mu\subset B_k$. Hence, for any $x\in B_k$,
  \[\int_{\mathbb{Q}_p^n} |x-y|_p^{-t}d\mu(y)= \sum_{m=-\infty}^{+\infty}\int_{\{y: |y-x|_p=p^m\}}|y-x|_p^{-t}d\mu(y)\le \sum_{m=-\infty}^{k}\frac{\mu(S(x,p^m))}{p^{mt}}.\]
  Note that $S(x,p^m)$ can be written as finitely many disjoint balls with radius $p^{m-1}$. That is, there exists some $i_0\ge 1$, such that
  \[S(x,p^m)=\bigsqcup_{i=1}^{i_0}B(x_i,p^{m-1}).\]
  Hence, by \eqref{frostman}, for any $x\in\mathbb{Q}_p^n$, 
  \[\mu(S(x,p^m))=\sum_{i=1}^{i_0}\mu(B(x_i,p^{m-1}))\lesssim p^{(m-1)s}.\] 
  Therefore, for any $0<t<s$,
  \begin{align*}
    I_t^p(\mu)=\int_{B_k}\int_{B_k} |x-y|_p^{-t}d\mu(x) d\mu(y) \lesssim &
    \int_{B_k} \sum_{m=-\infty}^{k}p^{m(s-t)}d\mu(x)\\
    \lesssim & \sum_{m=0}^{+\infty}p^{-m(s-t)}<+\infty.  
  \end{align*}
\end{proof}
\begin{lemma}\label{hwsl2}
  Let $A\subset \mathbb{Q}_p^n$ be a Borel set. If there exists $\mu\in \mathcal{M}(A)$, such that $I_s^p(\mu)<+\infty$ for some $0\le s\le n$, then there exists $\nu\in \mathcal{M}(A)$ such that \eqref{frostman} holds for $s$.
\end{lemma}
\begin{proof}
  Since $I_s^p(\mu)<+\infty$, we have
  \[\int_{\mathbb{Q}_p^n} |x-y|_p^{-s}d\mu(y)<+\infty\ \text{for}\ \mu\text{-}a.e.\  x\in A.\]
  Hence, we can find $M\in (0,+\infty)$, such that 
  \[C=\left\{x\in A: \int_{\mathbb{Q}_p^n} |x-y|_p^{-s}d\mu(y)<M\right\}\ \text{has positive}\ \mu\text{-}\text{measure}.\]
  Let $\nu=\mu|_C$. Then $\nu\in \mathcal{M}(A)$, and for any $x\in\mathbb{Q}_p^n$ and any $m\in\mathbb{Z}$,
  \begin{align*}
    \nu(B(x,p^m))= & \mu(B(x,p^m)\cap C)\\
    = & \int_{B(x,p^m)} \mathbf{1}_C(y)d\mu(y)\\
    \le & \int_{B(x,p^m)} \left(\frac{|y-x|_p}{p^m}\right)^{-s}\mathbf{1}_C(y)d\mu(y)\\
    \le & p^{ms}\int_C |y-x|_p^{-s}d\mu(y)<Mp^{ms}.
  \end{align*}
  This completes the proof.
\end{proof}
Now, we are ready to prove Theorem \ref{hwdl} using Lemmas \ref{Frostman}, \ref{hwsl1} and \ref{hwsl2}.
\begin{proof}[Proof of Theorem \ref{hwdl}]
  Denote the set $\{0\le s\le n:\exists\ \mu\in \mathcal{M}(A),\ \text{s.t.}\ \eqref{frostman}\ \text{holds for}\ s\}$ by $S$.
  We first prove that
  \begin{align}\label{hwd2}
    \dim_H(A)=\sup S.
  \end{align}
  For any $s\in S$, by Lemma \ref{Frostman}, we have $\mathcal{H}^s(A)>0$, which implies $\dim_H(A)\ge s$. Hence, $\dim_H(A)\ge \sup S$. \par
  On the other hand, suppose $\dim_H(A)>\sup S$. Then, we can choose $t$ such that $\sup S<t<\dim_H(A)$. By definition $\mathcal{H}^t(A)=+\infty$. Then, by Lemma \ref{Frostman}, we have $t\in S$, which leads to a contradiction.
  \par
  Let $T:=\{0\le s\le n:\ \text{exists}\ \mu\in \mathcal{M}(A),\ \text{such that}\ I_{s}^p(\mu)<+\infty\}$. To finish the proof, we show $\sup S=\sup T$. In fact, for any $s\in S$, by Lemma \ref{hwsl1}, if $0<t<s$, then $t\in T$. This implies $\sup T\ge \sup S$. On the other hand, for any $t\in T$, by Lemma \ref{hwsl2}, we have $t\in S$. Hence, $\sup T\le \sup S$.
\end{proof}
Before we prove Theorem \ref{thm1.2}, we need two lemmas.
\begin{lemma}\label{bjhs}
  Let $\mu\in\mathcal{M}(\mathbb{Q}_p^n)$. Let $\phi=\mathbf{1}_{\mathbb{Z}_p^n}$. For any $\varepsilon\in\mathbb{Q}_p\setminus \{0\}$, define $\phi_{\varepsilon}(x)=|\varepsilon|_p^{-n}\phi(\frac{x}{\varepsilon})$ for all $x\in\mathbb{Q}_p^n$. Let $\mu_{\varepsilon}=\phi_{\varepsilon} * \mu$. Then, $\mu_{\varepsilon}$ and $\widehat{\mu}_{\varepsilon}$ have compact support. Moreover,
  \[\mu_{\varepsilon},\ \widehat{\mu}_{\varepsilon}\in L^1(\mathbb{Q}_p^n).\]
\end{lemma}
\begin{proof}
  Since $\mu\in\mathcal{M}(\mathbb{Q}_p^n)$, then there exists some $k\in\mathbb{Z}$ such that $\text{spt}\mu\subset B_k$. Hence,
  \begin{align*}
    \mu_{\varepsilon}(x)=|\varepsilon|_p^{-n}\int_{B_k} \phi\left(\frac{x-y}{\varepsilon}\right)d\mu(y).
  \end{align*} 
  Therefore, if $|x|_p>\max\{|\varepsilon|_p,p^k\}$, then $\left|\frac{x-y}{\varepsilon}\right|_p>1$. This implies that $\mu_{\varepsilon}$ has compact support. On the other hand, for any $x\in \mathbb{Q}_p^n$,
  \[\mu_{\varepsilon}(x)=|\varepsilon|_p^{-n}\int_{\mathbb{Q}_p^n} \phi\left(\frac{x-y}{\varepsilon}\right)d\mu(y)=|\varepsilon|_p^{-n}\mu(B(x,|\varepsilon|_p))\le |\varepsilon|_p^{-n}\mu(\mathbb{Q}_p^n).\]
  Thus $\mu_{\varepsilon}\in L^1(\mathbb{Q}_p^n)$. Noting that 
  for any $y\in \mathbb{Q}_p^n$ with $|y|_p\le |\varepsilon|_p$, we have $\mu_{\varepsilon}(x+y)=\mu_{\varepsilon}(x)$. Then $\mu_{\varepsilon}$ is uniformly locally constant. By Proposition \ref{jbcz}, we have that $\widehat{\mu}_{\varepsilon}$ has compact support. Since $|\widehat{\mu}_{\varepsilon}|\le \| \mu_{\varepsilon}\|_1$, we conclude $\widehat{\mu}_{\varepsilon}\in L^1(\mathbb{Q}_p^n)$. This completes the proof.
\end{proof}
\begin{lemma}\label{cdbj}
  Let $\mu\in\mathcal{M}(\mathbb{Q}_p^n)$. Let $\phi$ be a complex valued function on $\mathbb{Q}_p^n$ with compact support. For any $\varepsilon\in\mathbb{Q}_p\setminus \{0\}$, define $\phi_{\varepsilon}(x)=|\varepsilon|_p^{-n}\phi(\frac{x}{\varepsilon})$. If $\int_{\mathbb{Q}_p^n} |\phi(x)|dx<+\infty$, then for any $x\in\mathbb{Q}_p^n$,
  \[\lim_{|\varepsilon|_p\to 0}\widehat{\mu}_{\varepsilon}(x)=\widehat{\phi}(0)\widehat{\mu}(x).\]
\end{lemma}
\begin{proof}
  Since $\phi$ has compact support, then there exists some $k\in\mathbb{Z}$, such that $\text{spt}\phi\subset B_k$.
  Since $\widehat{\phi}_{\varepsilon}(\xi)=\widehat{\phi}(\varepsilon\xi)$, we have
  \begin{align*}
    |\widehat{\phi}_{\varepsilon}(\xi)-\widehat{\phi}(0)|
     = & \left|\int \phi(x)\left(\chi(-\varepsilon\xi\cdot x)-1\right)dx\right|\\
     \le & \int |\phi(x)|\cdot 2\pi \{-\varepsilon\xi\cdot x\}_pdx\\
     \le & \int_{B_k} |\phi(x)|\cdot 2\pi |\varepsilon\xi\cdot x|_pdx\\
     = & 2\pi |\varepsilon|_p\int_{B_k} |\phi(x)||\xi\cdot x|_pdx.
  \end{align*}
  Hence,
  \begin{align}\label{a}
    \lim_{|\varepsilon|_p\to 0}\widehat{\phi}_{\varepsilon}(\xi)=\widehat{\phi}(0).
  \end{align} 
  Since $\phi_{\varepsilon}\in L^1(\mathbb{Q}_p^n)$ and $\mu\in\mathcal{M}(\mathbb{Q}_p^n)$, we have 
  $\widehat{\mu}_{\varepsilon}=\widehat{\phi}_{\varepsilon}\widehat{\mu}$. By \eqref{a}, we complete the proof.
\end{proof}
Now, we are at the position to prove Theorem \ref{thm1.2}.
\begin{proof}[Proof of Theorem \ref{thm1.2}]
  For $n\in \mathbb{N}$ and $s\in (0,n)$, let $C(n,s)=\frac{1-p^{-s}}{1-p^{s-n}}$ and define $K_s(x)=|x|_p^{-s}$ for $x\in\mathbb{Q}_p^n$. Note that $\sum\limits_{\gamma=0}^{+\infty}(p^{-\gamma})^{-s}p^{-n\gamma}=\sum\limits_{\gamma=0}^{+\infty}p^{(s-n)\gamma}<+\infty$. Then, by Lemma \ref{frf}, we have
  \[\widehat{K}_s(x)=|x|_p^{s-n}\left((1-p^{-n})\sum_{\gamma =0}^{+\infty}p^{(s-n)\gamma}-p^{-s}\right)=C(n,s)K_{n-s}(x),\ \forall x\neq 0.\]
  Let $\varphi$ be a positive real valued bounded function on $\mathbb{Q}_p^n$ with compact support. Suppose $\varphi,\ \widehat{\varphi}\in L^1(\mathbb{Q}_p^n)$. Then, we can define
  \[I_s^p(\varphi):=\int\int |x-y|_p^{-s}\varphi(x)\varphi(y)dxdy.\]
  Letting $z=y-x$ and denoting $\tilde{\varphi}(x)=\varphi(-x)$, we have 
  \begin{align*}
    I_s^p(\varphi) = & \int\int K_s(y-x)\varphi(x)\varphi(y)dxdy\\
    = & \int\int K_s(z)\varphi(y-z)\varphi(y)dzdy\\
    = & \int K_s(z)(\tilde{\varphi} * \varphi(z))dz.
  \end{align*}
  Since $\varphi,\ \widehat{\varphi}\in L^1(\mathbb{Q}_p^n)$, we have 
  \begin{align}\label{b}
    \widehat{|\widehat{\varphi}|^2}=\tilde{\varphi} * \varphi.
  \end{align}
  By assumption $\varphi$ has compact support. Hence, $I_s^p(\varphi)<+\infty$. By \eqref{b} and Fubini's theorem,
  \[\int K_s(z)(\tilde{\varphi} * \varphi(z))dz=\int \widehat{K}_s(z)|\widehat{\varphi}(z)|^2dz.\]
  Hence,
  \[I_s^p(\varphi)=C(n,s)\int |\widehat{\varphi}(x)|^2|x|_p^{s-n}dx.\]
  For any $\varepsilon\in\mathbb{Q}_p\setminus \{0\}$, define $\phi_{\varepsilon}(x)=|\varepsilon|_p^{-n}\phi(\frac{x}{\varepsilon})$ with $\phi(x)=\mathbf{1}_{\mathbb{Z}_p^n}(x),\ \forall x\in \mathbb{Q}_p^n$. Let $\mu_{\varepsilon}:=\phi_{\varepsilon} * \mu$. By Lemma \ref{bjhs}, we have that $\mu_{\varepsilon}$ has compact support and $\mu_{\varepsilon},\ \widehat{\mu}_{\varepsilon}\in L^1(\mathbb{Q}_p^n)$. Therefore, we can choose $\varphi=\mu_{\varepsilon}$, and we have
  \begin{align*}
    I_s^p(\mu_{\varepsilon})=C(n,s)\int |\widehat{\mu}_{\varepsilon}(x)|^2|x|_p^{s-n}dx.
  \end{align*}
  By Lemma \ref{cdbj} and the dominated convergence theorem, 
  \[\lim_{|\varepsilon|_p\to 0}I_s^p(\mu_{\varepsilon})=C(n,s)\int |\widehat{\mu}(x)|^2|x|_p^{s-n}dx.\]
  On the other hand, by the definition of $I_s^p(\mu_{\varepsilon})$ and Fubini's theorem,
  \begin{align*}
    I_s^p(\mu_{\varepsilon})= & C(n,s)\int\int\left(|x-y|_p^{-s}\int \phi_{\varepsilon}(x-z)d\mu(z)\int \phi_{\varepsilon}(y-\omega)d\mu(\omega)\right)dxdy\\
    = & C(n,s)\int\int\left(\int\int |x-y|_p^{-s}\phi_{\varepsilon}(x-z)\phi_{\varepsilon}(y-\omega)dxdy\right)d\mu(z)d\mu(\omega).
  \end{align*}
  Now, we consider the integral
  \begin{align}\label{c}
    \int\int |x-y|_p^{-s} \phi_{\varepsilon}(x-z) \phi_{\varepsilon}(y-\omega)dxdy.
  \end{align}
  By the change of variables $u=\frac{x-z}{\varepsilon},\ v=\frac{y-\omega}{\varepsilon}$, the integral \eqref{c} becomes 
  \[\int\int |\varepsilon(u-v)+z-\omega|_p^{-s} \phi(u) \phi(v)dudv.\]
  In order to prove $\eqref{c}\le |z-\omega|_p^{-s}$, we distinguish three cases.
  \begin{enumerate}[i)]
    \item If $|z-\omega|_p<|\varepsilon(u-v)|_p$, then $|\varepsilon(u-v)+z-\omega|_p^{-s}=|\varepsilon(u-v)|_p^{-s}<|z-\omega|_p^{-s}$. Hence,
  \[\int\int |x-y|_p^{-s} \phi_{\varepsilon}(x-z) \phi_{\varepsilon}(y-\omega)dxdy<|z-\omega|_p^{-s}.\]
  \item If $|z-\omega|_p>|\varepsilon(u-v)|_p$, then 
  $|\varepsilon(u-v)+z-\omega|_p^{-s}=|z-\omega|_p^{-s}$. Hence,
  \[\int\int |x-y|_p^{-s} \phi_{\varepsilon}(x-z) \phi_{\varepsilon}(y-\omega)dxdy=|z-\omega|_p^{-s}.\]
  \item If $|z-\omega|_p=|\varepsilon(u-v)|_p$, then by noting that $du, dv$ are Haar measure, 
  \[\int\int |\varepsilon(u-v)+z-\omega|_p^{-s} \phi(u) \phi(v)dudv\lesssim |z-\omega|_p^{-s}.\]
  Therefore, we have
  \begin{align}\label{d}
    \int\int |x-y|_p^{-s} \phi_{\varepsilon}(x-z) \phi_{\varepsilon}(y-\omega)dxdy\lesssim |z-\omega|_p^{-s}.
  \end{align}
  \end{enumerate}
  Now, we distinguish two cases.
  \begin{enumerate}[i)]
    \item If $I_s^p(\mu)<+\infty$, then by \eqref{d} and the dominated convergence theorem,
    \begin{align*}
      \lim_{|\varepsilon|_p\to 0}I_s^p(\mu_{\varepsilon})= &\int\int \left(\lim_{|\varepsilon|_p\to 0}\int\int |x-y|_p^{-s} \phi_{\varepsilon}(x-z) \phi_{\varepsilon}(y-\omega)dxdy\right)d\mu(z)d\mu(\omega)\\
      = & \int\int |z-\omega|_p^{-s}d\mu(z)d\mu(\omega)\\
      = & I_s^p(\mu).
    \end{align*}
    \item If $I_s^p(\mu)=+\infty$, then by Fatou's lemma,
    \begin{align*}
      I_s^p(\mu)\le \underset{|\varepsilon|_p\to 0}{\varliminf}I_s^p(\mu_{\varepsilon})=\int |\widehat{\mu}(x)|^2|x|_p^{s-n}dx.
    \end{align*}
  \end{enumerate}
  Combining the above two cases, we complete the proof.
\end{proof}
Now, we prove Theorem \ref{qiuhe} by using Lemma \ref{bjhs} and Parseval's identity.
\begin{proof}[Proof of Theorem \ref{qiuhe}]
  Define $\mu_{\varepsilon}$ as in Lemma \ref{bjhs} for $\varepsilon\in\mathbb{Z}_p\setminus \{0\}$. Since 
  $$B(x,|\varepsilon|_p)\subset S(0,|x|_p),\ \forall\ |x|_p>1,$$
  Lemma \ref{bjhs} implies that
  \begin{align}\label{e}
    \text{spt}\mu_{\varepsilon}\subset \mathbb{Z}_p^n.
  \end{align}
  We claim that 
  $$(\mathbf{1}_{\mathbb{Z}_p^n}\cdot K_s)* \mu_{\varepsilon}\in L^2(\mathbb{Z}_p^n).$$
  Then, by Parseval's identity and the fact that $\mathbf{1}_{\mathbb{Z}_p^n}K_s,\ \mu_{\varepsilon}\in L^1(\mathbb{Z}_p^n)$, we obtain
  \[\int_{\mathbb{Z}_p^n}((\mathbf{1}_{\mathbb{Z}_p^n}\cdot K_s)* \mu_{\varepsilon}(x))\mu_{\varepsilon}(x)dx=\underset{z\in\widehat{\mathbb{Z}}_p^n}{\sum}\widehat{\mathbf{1}_{\mathbb{Z}_p^n}\cdot K_s}(z)|\widehat{\mu}_{\varepsilon}(z)|^2.\]
  For any $z\in \widehat{\mathbb{Z}}_p^n\setminus \{0\}$, we have 
  \begin{align*}
    \widehat{\mathbf{1}_{\mathbb{Z}_p^n}\cdot K_s}(z) & = \int_{\mathbb{Z}_p^n}|x|_p^{-s}\chi_p(-x\cdot z)dx\\
    & = C(n,s)|z|_p^{s-n}.
  \end{align*}
  Therefore, 
  \[\lim_{|\varepsilon|_p\to 0}\int_{\mathbb{Z}_p^n}((\mathbf{1}_{\mathbb{Z}_p^n}\cdot K_s)* \mu_{\varepsilon}(x))\mu_{\varepsilon}(x)dx=\frac{1-p^{-n}}{1-p^{s-n}}\mu(\mathbb{Z}_p^n)^2+\frac{1-p^{-s}}{1-p^{s-n}}\underset{z\in\widehat{\mathbb{Z}}_p^n\setminus \{0\}}{\sum}|\widehat{\mu}(z)|^2|z|_p^{s-n}.\]
  On the other hand, by the fact that $\mathbb{Z}_p^n-\mathbb{Z}_p^n=\mathbb{Z}_p^n$ and Fubini's theorem, we have
  \begin{align*}
    \int_{\mathbb{Z}_p^n}((\mathbf{1}_{\mathbb{Z}_p^n}\cdot K_s)* \mu_{\varepsilon}(x))\mu_{\varepsilon}(x)dx & = \int_{\mathbb{Z}_p^n}\int_{\mathbb{Z}_p^n}\mathbf{1}_{\mathbb{Z}_p^n}(x-y)|x-y|_p^{-s} \mu_{\varepsilon}(y)\mu_{\varepsilon}(x)dxdy\\
    & = I_s^p(\mu_{\varepsilon}).
  \end{align*}
  This completes the proof.\par 
  We now prove the claim. Since $0<s<n$, we have
  \[\int_{\mathbb{Z}_p^n}|x-y|_p^{-s}dy=\int_{\mathbb{Z}_p^n}|y|_p^{-s}dy=\frac{1-p^{-n}}{1-p^{s-n}}<+\infty.\]
  Moreover, by \eqref{e}, we have that the convolution $(\mathbf{1}_{\mathbb{Z}_p^n}\cdot K_s)* \mu_{\varepsilon}$ is supported in $\mathbb{Z}_p^n$.
  Therefore,
  \[\int_{\mathbb{Q}_p^n}\left|\int_{\mathbb{Q}_p^n}\mathbf{1}_{\mathbb{Z}_p^n}(x-y)|x-y|_p^{-s} \mu_{\varepsilon}(y)dy\right|^2dx\lesssim \int_{\mathbb{Z}_p^n}\left|\int_{\mathbb{Z}_p^n}|x-y|_p^{-s}dy\right|^2dx<+\infty.\]
  Consequently, 
  $$(\mathbf{1}_{\mathbb{Z}_p^n}\cdot K_s)* \mu_{\varepsilon}\in L^2(\mathbb{Z}_p^n).$$
\end{proof}
The proof of Theorem \ref{thm1.3} follows the same argument as its Euclidean counterpart. Indeed, having established Theorems \ref{hwdl} and \ref{thm1.2}, the proof proceeds analogously to that in $\mathbb{R}^n$.
\begin{proof}[Proof of Theorem \ref{thm1.3}]
  Let $t\in[0,n]$ be such that there exists a measure $\mu\in\mathcal{M}(A)$ satisfying
  \begin{equation}\label{eq:mu-decay}
   |\widehat{\mu}(x)|\lesssim |x|_p^{-\frac{t}{2}},\qquad \forall\, x\in\mathbb{Q}_p^n.
  \end{equation}
  If $\dim_F(A)=0$, then theorem \ref{thm1.3} holds obviously. Hence, without loss of generality, we assume $t>0$.
  Since $\mu\in\mathcal{M}(A)$, there exists an integer $k\in\mathbb{Z}$ such that $\operatorname{spt}\mu\subset B_k$. Consequently,
 \[\widehat{\mu}(x)=\int_{\mathbb{Q}_p^n} \chi(-x\cdot y)\,d\mu(y)=\int_{B_k}\chi(-x\cdot y)\,d\mu(y).\]
  If $|x|_p\le p^{-k}$, then for every $y\in B_k$ we have $|x\cdot y|_p\le |x|_p|y|_p\le 1$, which implies $\{x\cdot y\}_p=0$. Therefore,
  \begin{align*}
  \widehat{\mu}(x)=\mu(B_k)=\mu(\mathbb{Q}_p^n),\qquad \forall\, x\in B_{-k}.
  \end{align*}
  Now, let $s\in(0,t)$ be arbitrary. The $s$-dimensional $p$-adic energy of $\mu$ can be estimated as follows:
  \begin{align}\label{eq:energy-split}
   I_s^p(\mu) 
   &= \frac{1-p^{-s}}{1-p^{s-n}}\int_{\mathbb{Q}_p^n} |\widehat{\mu}(x)|^2|x|_p^{s-n}\,dx \notag\\[4pt]
   &\lesssim \sum_{m=-\infty}^{+\infty}\int_{S_m}|\widehat{\mu}(x)|^2|x|_p^{s-n}\,dx \notag\\[4pt]
   &= \sum_{m=-\infty}^{-k}\int_{S_m}\bigl(\mu(\mathbb{Q}_p^n)\bigr)^2 p^{m(s-n)}\,dx
   +\sum_{m=-k+1}^{+\infty}\int_{S_m}|\widehat{\mu}(x)|^2 p^{m(s-n)}\,dx \notag\\[4pt]
   &= \bigl(\mu(\mathbb{Q}_p^n)\bigr)^2(1-p^{-n})\sum_{m=k}^{+\infty}p^{-sm}
   +\sum_{m=-k+1}^{+\infty}\int_{S_m}|\widehat{\mu}(x)|^2 p^{m(s-n)}\,dx. 
  \end{align}
  For the second sum in \eqref{eq:energy-split}, we use the decay hypothesis \eqref{eq:mu-decay} to obtain
  \[\sum_{m=-k+1}^{+\infty}\int_{S_m}|\widehat{\mu}(x)|^2 p^{m(s-n)}\,dx
  \lesssim \sum_{m=-k+1}^{+\infty}\int_{S_m} p^{(s-t-n)m}\,dx
  \lesssim \sum_{m=0}^{+\infty}p^{(s-t)m} <+\infty,\]
where the last series converges since $s<t$. The first sum in \eqref{eq:energy-split} is a convergent geometric series as well, because $s> 0$. Hence $I_s^p(\mu)<+\infty$ for all $s\in(0,t)$. Applying Theorem~\ref{hwdl}, we conclude that $\dim_H(A)\ge t$, which completes the proof.
\end{proof}
\section{Fourier dimension of $p$-adic Riesz products}
We first consider the case that $\widehat{\mu}_{p,a,\ell}(\xi)$ is almost a monomial for all $|\xi|_p$ large enough.
\begin{proposition}\label{rpfws1}
  If $p\ge 3$ or $p=2$ and $\ell_{i+1}-\ell_{i}\ge 2$ for all $i$ large enough, then
  \[\dim_F(\mu_{p,a,\ell})=2\underset{n\to +\infty}{\varliminf} \frac{\log(|a_n|)}{-\ell_n \log(p)}.\]
\end{proposition}
\begin{proof}
Let $\dim_F(\mu_{p,a,\ell})=s$. We first establish the lower bound
\begin{equation}\label{eq:lower-bound}
\underset{n\to +\infty}{\varliminf}\frac{\log |a_n|}{-\ell_n\log p}\ge \frac{s}{2}.
\end{equation}
By the definition of Fourier dimension, for every $\varepsilon>0$ there exists a constant $C_\varepsilon>0$ such that
\begin{align*}
|\widehat{\mu}_{p,a,\ell}(\xi)|\le C_\varepsilon\,|\xi|_p^{-\frac{s-\varepsilon}{2}},\qquad \forall\,\xi\in\widehat{\mathbb{Z}}_p.
\end{align*}
We now distinguish two cases.
\begin{enumerate}[i)]
  \item Suppose $p\ge 3$. Set $\xi_n=p^{-\ell_n}$ for each $n\in\mathbb{N}$. Then, by \eqref{rpfxs}, we have
  \[\widehat{\mu}_{p,a,\ell}(\xi_n)=\frac{a_n}{2}.\]
  \item Suppose $p=2$ and $\ell_{i+1}-\ell_i\ge 2$ for all sufficiently large $i$. There exists $k\in\mathbb{N}$ such that $\ell_{n+1}\ge \ell_n+2$ for all $n\ge k$. For each $n\ge k$, set $\xi_n=2^{-\ell_n}$. Then by \eqref{rpfxs}, we have
  \[\widehat{\mu}_{2,a,\ell}(\xi_n)=\frac{a_n}{2}.\]
\end{enumerate}
In either case, for all sufficiently large $n$ we have
\[
\frac{|a_n|}{2}\,|\xi_n|_p^{\frac{s-\varepsilon}{2}}
=|\widehat{\mu}_{p,a,\ell}(\xi_n)|\,|\xi_n|_p^{\frac{s-\varepsilon}{2}}
\le C_\varepsilon.
\]
Since $|\xi_n|_p=p^{\ell_n}$, this yields
\[
|a_n|\le 2C_\varepsilon\, p^{-\ell_n\cdot\frac{s-\varepsilon}{2}},
\]
and consequently
\[
\frac{\log |a_n|}{-\ell_n\log p}\ge \frac{s-\varepsilon}{2}-\frac{\log(2C_\varepsilon)}{\ell_n\log p}.
\]
Taking the limit inferior as $n\to\infty$ gives
\[
\underset{n\to +\infty}{\varliminf}\frac{\log |a_n|}{-\ell_n\log p}\ge \frac{s-\varepsilon}{2}.
\]
Letting $\varepsilon\to 0^+$, we obtain \eqref{eq:lower-bound}.
We now prove the reverse inequality
\begin{equation}\label{eq:upper-bound}
\underset{n\to +\infty}{\varliminf}\frac{\log |a_n|}{-\ell_n\log p}\le \frac{s}{2}
\end{equation}
by contradiction. Suppose that
\[
\underset{n\to +\infty}{\varliminf}\frac{\log |a_n|}{-\ell_n\log p}>\frac{s}{2}.
\]
Then, there exist $\varepsilon>0$ and $N\in\mathbb{N}$ such that
\begin{equation}\label{eq:an-contra}
|a_n|<\frac{1}{2}\,p^{-\ell_n\cdot\frac{s+\varepsilon}{2}},\qquad \forall\,n\ge N.
\end{equation}
For any $\xi\in S_{\ell_n}$ with $n\ge N$, it follows from \eqref{rpfxs} that
\[
|\widehat{\mu}_{p,a,\ell}(\xi)|\le \frac{|a_n|}{2}.
\]
Combining this with \eqref{eq:an-contra} and using $|\xi|_p=p^{\ell_n}$ for $\xi\in S_{\ell_n}$, we obtain
\[
|\widehat{\mu}_{p,a,\ell}(\xi)|<\frac{1}{4}\,|\xi|_p^{-\frac{s+\varepsilon}{2}}.
\]
Since $|\widehat{\mu}_{p,a,\ell}(\xi)|=0$ for all $\xi\notin \bigcup_{n=1}^{+\infty} S_{\ell_n}$, we have $\dim_F(\mu_{p,a,\ell})\ge s+\varepsilon>s$. This leads to a contradiction with the definition of $s$. Therefore, we establish \eqref{eq:upper-bound} and complete the proof.
\end{proof}
We now prove Corollary \ref{jdlx} using Propositions \ref{rpac} and \ref{rpfws1}.
\begin{proof}[Proof of Corollary \ref{jdlx}]
  Let $\dim_F(\mu_{p,a,\ell})=a\ \text{and}\ \dim_F(\mu_{p,b,\ell})=b$ with $a\le b$. Since $a,\ b>0$, by Proposition \ref{rpfws1}, there exists a $K\in\mathbb{N}$ such that $|a_k|\le p^{-\frac{a\ell_k}{2}}$ and $|b_k|\le p^{-\frac{b\ell_k}{2}}$ for any $k\ge K$. Note that 
  \[p^{\ell_k-\ell_{k-1}}(2-|a_k+b_k|)^2\ge p(2-|a_k|-|b_k|)^2.\]
  Then,
  \[\lim_{k\to +\infty}p^{\ell_k-\ell_{k-1}}(2-|a_k+b_k|)^2\ge 4p.\]
  Therefore, we have $p^{\ell_k-\ell_{k-1}}(2-|a_k+b_k|)^2>2p$ for all $k$ large. Since
  \begin{align*}
    \sum_{k=1}^{+\infty}|a_k-b_k|^2\left(1+\frac{\cos ^2(s_k-t_k)}{\sqrt{2-|a_k+b_k|}}\right) \le & \sum_{k=1}^{+\infty}(|a_k|+|b_k|)^2\left(1+\frac{1}{\sqrt{2-|a_k|-|b_k|}}\right)\\
    \le & \sum_{k=1}^{+\infty} 4p^{-ak}\left(1+\frac{1}{\sqrt{2-2p^{-ak}}}\right)<+\infty,
  \end{align*}
  by Proposition \ref{rpac}, we complete the proof.
\end{proof}
When $p=2$ and $\ell_{i+1}=\ell_i+1$ for infinitely many $i$, the main difficulty lies in obtaining a uniform estimate for $|\widehat{\mu}_{2,a,\ell}(\xi)|$ for all sufficiently large $|\xi|_p$. Indeed, the Fourier transform $\widehat{\mu}_{2,a,\ell}(\xi)$ is represented as a sum of complex numbers, and the number of summands may tend to infinity as $|\xi|_p\to+\infty$. \par
For any $m\ge 1$ and $k\in\mathbb{R}$, let
  \[E_m(k)=\left\{(\varepsilon_1,\cdots ,\varepsilon_m)\in \{-1,0,1\}^m: \sum_{i=1}^{m}\varepsilon_i2^{-i}=k\right\}.\]
Now, we prove some propositions of $E_m(k)$ which play an important role in the proof of Theorem \ref{thm1.5}.
\begin{proposition}\label{cd}
  Suppose $k=\sum_{i=1}^{n}c_i2^{-i}$ with $c_i\in \{0,1\}$ for all $1\le i\le n$ and $c_n=1$. Then, we have 
  \begin{enumerate}[i)]
  \item If $m\le n-1$, then $E_m(k)=\emptyset$.
  \item If $m\ge n$, then every $(\varepsilon_1,\dots,\varepsilon_m)\in E_m(k)$ satisfies
  \[
  |\varepsilon_n|=1\quad\text{and}\quad\varepsilon_j=0\quad \text{for all}\ j\ge n+1.
  \]
  \item $\#E_m(2^{-n})=n$ for every $m\ge n$.
\end{enumerate}
\end{proposition}
\begin{proof}
  (i) Suppose $(\varepsilon_1,\dots,\varepsilon_m)\in E_m(k)$. By the definition of $E_m(k)$, we have $2^mk\in \mathbb{Z}$. On the other hand,
  \[
    2^m k=\sum_{i=m+1}^{n}c_i2^{m-i}+\sum_{i=1}^{m}c_i2^{m-i}.
  \]
 Since $n-m\ge 1$, we have 
 \[0<2^{m-n}\le\sum_{i=m+1}^{n}c_i2^{m-i}\le \sum_{i=1}^{n-m}2^{-i}<1.\] 
 Therefore, $2^mk\notin \mathbb{Z}$. This leads to a contradiction.\\
 (ii) Let $m\ge n$ and $(\varepsilon_1,\cdots,\varepsilon_m)\in E_m(k)$. Multiplying by $2^n$ gives
\begin{equation*}
\sum_{i=1}^{m}\varepsilon_i2^{n-i}=1+\sum_{i=1}^{n-1}c_i2^{n-i}.
\end{equation*}
Note that the right-hand side is odd. Therefore,
\[\varepsilon_n+\sum_{i=n+1}^{m}\varepsilon_i2^{n-i}\]
should be odd too.
Since 
\[\left|\sum_{i=n+1}^{m}\varepsilon_i2^{n-i}\right|\le\sum_{j=1}^{m-n}2^{-j}<1,\] 
we have $|\varepsilon_n|=1$ and $\sum_{i=n+1}^{m}\varepsilon_i2^{n-i}=0$.

Suppose there exists some $\varepsilon_j\neq0$ for $j\ge n+1$. Then, we can choose $j_0$ to be the largest such index. Hence,
\[
0=\varepsilon_{j_0}+\sum_{i=n+1}^{j_0-1}\varepsilon_i2^{j_0-i}.
\]
Therefore, we have that $\varepsilon_{j_0}$ is even, which leads to a contradiction that $\varepsilon_{j_0}\in\{-1,1\}$. Therefore $\varepsilon_j=0$ for all $j\ge n+1$.
\\
(iii) Take $k=2^{-n}$. By (ii), any $(\varepsilon_1,\cdots,\varepsilon_m)\in E_m(2^{-n})$ satisfies $|\varepsilon_n|=1$ and $\varepsilon_j=0$ for $j>n$.

If $\varepsilon_n=1$, then $\sum_{i=1}^{n-1}\varepsilon_i2^{-i}=0$, and the same parity argument as in (ii) yields $\varepsilon_1=\dots=\varepsilon_{n-1}=0$. This gives the solution $v^{(n)}=(0,\dots,0,1,0,\dots,0)$.

If $\varepsilon_n=-1$, then $\sum_{i=1}^{n-1}\varepsilon_i2^{-i}=2^{-(n-1)}$. Repeating the argument recursively, for each $j\in\{1,\dots,n\}$ there is exactly one solution
\[
v^{(j)}=(\underbrace{0,\dots,0}_{j-1},\,1,\,\underbrace{-1,\dots,-1}_{n-j},\,0,\dots,0),
\]
where the entry $1$ occurs in position $j$ and all subsequent entries up to position $n$ equal $-1$. Direct verification shows
\[
\sum_{i=1}^{n}v^{(j)}_i2^{-i}=2^{-j}-\sum_{i=j+1}^{n}2^{-i}=2^{-n}.
\]
These $n$ solutions are distinct, and the recursive analysis shows there are no others. Hence, $\#E_m(2^{-n})=n$ for all $m\ge n$.
\end{proof}
The following lemma provides an upper bound for any $|\widehat{\mu}_{2,a,\ell}(\xi)|$ for $|\xi|_p\in S_{\ell_n}$.
\begin{lemma}\label{upper bound of measure}
 For any $n\ge 1$, we have 
 \[\left|\widehat{\mu}_{2,a,\ell}(\xi)\right|\le |a_n|,\ \forall \xi\in S_{\ell_n}.\]
\end{lemma}
\begin{proof}
  By \eqref{rpfxs}, without loss of generality, we can assume that $\xi=\sum_{i=1}^{n}c_i2^{-\ell_i}$ with $c_i\in \{0,1\}$ for all $1\le i\le n$ and $c_n=1$. By Lemma \ref{cd} and \eqref{rpfxs}, we have 
  \[\widehat{\mu}_{2,a,\ell}(\xi)=\underset{(\varepsilon_1,\cdots,\varepsilon_{\ell_n})\in E_{\ell_n}(\xi)}{\sum}\prod_{i=1}^{n}\left(\frac{a_i}{2}\right)^{(\varepsilon_{\ell_i})}.\]
  By the definition of $a^{(\varepsilon)}$, we have $|a^{(\varepsilon)}|=|a|^{|\varepsilon|}$.
  Hence,
  \[\left|\widehat{\mu}_{2,a,\ell}(\xi)\right|\le\underset{(\varepsilon_1,\cdots,\varepsilon_{\ell_n})\in E_{\ell_n}(\xi)}{\sum}\prod_{i=1}^{n}\left(\frac{|a_i|}{2}\right)^{|\varepsilon_{\ell_i}|}.\]
  For any $(\varepsilon_1,\cdots,\varepsilon_{\ell_n})\in E_{\ell_n}(\xi)$, by Lemma \ref{cd}, we have $|\varepsilon_{\ell_n}|=1$. Therefore,
  \[\underset{(\varepsilon_1,\cdots,\varepsilon_{\ell_n})\in E_{\ell_n}(\xi)}{\sum}\prod_{i=1}^{n}\left(\frac{|a_i|}{2}\right)^{|\varepsilon_{\ell_i}|}=\frac{|a_n|}{2}\underset{(\varepsilon_1,\cdots,\varepsilon_{\ell_n-1})\in E_{\ell_n-1}(\xi_1)\cup E_{\ell_n-1}(\xi_2)}{\sum}\prod_{i=1}^{n-1}\left(\frac{|a_i|}{2}\right)^{|\varepsilon_{\ell_i}|},\]
  where $\xi_1=\xi-2^{\ell_n}$ and $\xi_2=\xi+2^{\ell_n}$.
  Let $\{a_k^{\prime}\}$ be a sequence with $a_k^{\prime}=|a_k|$ for all $k\ge 1$. Note that for any $j\in \{1,2\}$, we have
  \[\underset{(\varepsilon_1,\cdots,\varepsilon_{\ell_n-1})\in E_{\ell_n-1}(\xi_j)}{\sum}\prod_{i=1}^{n-1}\left(\frac{|a_i|}{2}\right)^{|\varepsilon_{\ell_i}|}=|\mu_{2,a^{\prime},\ell}(\xi_i)|\le 1.\]
  Therefore, we have 
  \[\left|\widehat{\mu}_{2,a,\ell}(\xi)\right|\le \frac{|a_n|}{2}\cdot (1+1)=|a_n|.\]
\end{proof}
Now, we are at the position to prove the lower bound of $\dim_F(\mu_{2,a,\ell})$.
\begin{proposition}\label{proof of lower bound}
  \[\dim_F(\mu_{2,a,\ell})\ge 2\underset{n\to +\infty}{\varliminf} \frac{\log(|a_n|)}{-\ell_n \log(2)}.\]
\end{proposition}
\begin{proof}
  By Lemma \ref{proof of lower bound}, for any $\xi\in S_{\ell_n}$, we have 
  \[\left|\widehat{\mu}_{2,a,\ell}(\xi)\right|\le |a_n|.\]
  By the same argument in the proof of the lower bound in Proposition \ref{rpfws1}, we complete the proof.
\end{proof}
In order to prove the upper bound of $\dim_F(\mu_{2,a,\ell})$, we need to give a lower bound of $\widehat{\mu}_{2,a,\ell}(\xi)$ for some $\xi$.
\begin{lemma}\label{lbof}
  If $\underset{n\to +\infty}{\varliminf} |a_n|<1$, then there exists a sequence $\{t_k\}\subset \mathbb{N}$ such that 
  \[|\widehat{\mu}_{2,a,\ell}(2^{-\ell_{t_k}})|\gtrsim |a_{t_k}|,\ \text{for any}\ k\ \text{large enough}.\]
\end{lemma}
\begin{proof}
  Since $\varliminf\limits_{n\to+\infty}|a_n|<1$, there exists a subsequence $\{n_k\}$ and a constant $a\in(0,1)$ such that $|a_{n_k}|\le a$ for all sufficiently large $k$. Set $t_k=n_k+1$, so that $|a_{t_k-1}|=|a_{n_k}|\le a$.

By part (iii) of Lemma~\ref{cd}, we have
\[
|\widehat{\mu}_{2,a,\ell}(2^{-\ell_{t_k}})|
=\left|\frac{a_{t_k}}{2}+\sum_{j=1}^{t_k-1}\lambda(j)\,
\frac{\overline{a_{t_k}}}{2}\frac{\overline{a_{t_k-1}}}{2}\cdots
\frac{\overline{a_{t_k-j+1}}}{2}\frac{a_{t_k-j}}{2}\right|,
\]
where $\lambda(j)=1$ if $\ell_{t_k-i}=t_k-i$ for all $1\le i\le j$, and $\lambda(j)=0$ otherwise.

Factor out $\frac{|a_{t_k}|}{2}$ and use the triangle inequality to obtain, for all large $k$,
\begin{align*}
|\widehat{\mu}_{2,a,\ell}(2^{-\ell_{t_k}})|
&=\frac{|a_{t_k}|}{2}\left|1+\sum_{j=1}^{t_k-1}\lambda(j)\,
\frac{\overline{a_{t_k-1}}}{2}\cdots
\frac{\overline{a_{t_k-j+1}}}{2}\frac{a_{t_k-j}}{2}\right|\\[4pt]
&\ge\frac{|a_{t_k}|}{2}\left(1-\sum_{j=1}^{t_k-1}
\frac{|a_{t_k-1}|}{2}\cdots
\frac{|a_{t_k-j+1}|}{2}\frac{|a_{t_k-j}|}{2}\right)\\[4pt]
&\ge\frac{|a_{t_k}|}{2}\left(1-|a_{t_k-1}|\sum_{j=1}^{t_k-1}\frac{1}{2^{\,j}}\right)\\[4pt]
&>\frac{|a_{t_k}|}{2}\bigl(1-|a_{t_k-1}|\bigr)
\ge\frac{1-a}{2}\,|a_{t_k}|.
\end{align*}
This completes the proof.
\end{proof}
\begin{lemma}\label{ybfs}
  Let $\{a_i\}\subset\mathbb{C}$ satisfy $0<|a_i|\le 1$ for all $i\ge 1$. Then there exist infinitely many positive integers $n$ such that
  \begin{align}\label{lb1}
    \left|\frac{a_n}{2}+\frac{\overline{a_n}}{2} \frac{a_{n-1}}{2}+\cdots+\frac{\overline{a_{n}}}{2}\frac{\overline{a_{n-1}}}{2}\cdots \frac{a_1}{2}\right|\gtrsim |a_n|.
  \end{align}
\end{lemma}
\begin{proof}
  For each $n\ge 1$, set
\[S_n:=\frac{a_n}{2}+\frac{\overline{a_n}}{2}\frac{a_{n-1}}{2}+\cdots+\frac{\overline{a_n}}{2}\frac{\overline{a_{n-1}}}{2}\cdots\frac{a_1}{2}.\]
  We distinguish two cases.
  \begin{enumerate}[i)]
    \item Suppose $\underset{n\to +\infty}{\varliminf}\frac{1}{|a_n|}\left|S_n\right|>0$.
    Then, there exists a constant $c>0$ and a sequence $\{n_k\}$, such that 
    \[\lim_{n\to +\infty}\frac{1}{|a_{n_k}|}\left|S_{n_k}\right|=c.\]
    Hence \eqref{lb1} holds for infinitely many $n$.
    \item Suppose $\underset{n\to +\infty}{\varliminf}\frac{1}{|a_n|}\left|S_n\right|=0$. Then, there exists a sequence $\{n_k\}$, such that for any $k$ large enough
    \[\left|S_{n_k}\right|\le \frac{|a_{n_k}|}{2}\le \frac{1}{2}.\]
    Hence, by the same calculation as in the proof of Lemma \ref{lbof}, we obtain
    \[\left|S_{n_k+1}\right|\ge \frac{|a_{n_k+1}|}{4}\]
    for all sufficiently large $k$. Since $\{n_k+1\}$ is an infinite sequence, this completes the proof.
  \end{enumerate}
\end{proof}
Now, we are at the position to prove the upper bound of $\dim_F(\mu_{2,a,\ell})$.
\begin{proposition}\label{flyws2}
  \[\dim_F(\mu_{2,a,\ell})\le 2\underset{n\to +\infty}{\varliminf}\frac{\log(|a_n|)}{-\ell_n \log(2)}.\]
\end{proposition}
\begin{proof}
  First, we claim that 
  \begin{align}\label{bad up}
    \dim_F(\mu_{2,a,\ell})\le 2\underset{n\to +\infty}{\varlimsup}\frac{\log(|a_n|)}{-\ell_n \log(2)}.
  \end{align}
  The proof of \eqref{bad up} proceeds by a case distinction on the sequence $\{\ell_i\}$.
  \begin{enumerate}[i)]
    \item Suppose $\ell_{i+1}=\ell_{i}+1$ for any $i$ large enough. Then, for any $n$ sufficiently large, we have
    \[|\widehat{\mu}_{2,a,\ell}(2^{-n})|\gtrsim \left|\frac{a_n}{2}+\frac{\overline{a_n}}{2} \frac{a_{n-1}}{2}+\cdots+\frac{\overline{a_{n}}}{2}\frac{\overline{a_{n-1}}}{2}\cdots \frac{a_1}{2}\right|.\]
    By Lemma \ref{lb1}, there exist infinitely many positive integers $n$ such that 
    \[|\widehat{\mu}_{2,a,\ell}(2^{-n})|\gtrsim |a_n|.\]
    Therefore, by the same argument in the proof of Proposition \ref{rpfws1}, the inequality \eqref{bad up} holds.
    \item Suppose there exists infinitely many $n$, such that $\ell_{n+1}-\ell_n\ge 2$. Then, there exists a sequence $\{n_k\}$ such that $\ell_{n_k+1}\ge \ell_{n_k}+2$. Let $\xi_k=2^{-\ell_{n_k+1}}$. Then, 
    \[|\widehat{\mu}_{2,a,\ell}(\xi_k)|=\frac{|a_{n_k+1}|}{2}.\]
    Hence, by the same argument in the proof of Proposition \ref{rpfws1}, the inequality \eqref{bad up} holds.
  \end{enumerate}
  Now, we prove Proposition \ref{flyws2} by a case distinction on the sequence $\{a_i\}$.
  \begin{enumerate}[i)]
    \item Suppose $\underset{n\to +\infty}{\lim}|a_n|=0$. Then, we can set $t_k=k-1$ in the proof of Lemma \ref{lbof}. Hence, for all sufficiently large $k$, we have   
    \[|\widehat{\mu}_{2,a,\ell}(2^{-\ell_{k}})|\gtrsim |a_{k}|.\]
    Therefore, by the same argument in the proof of Proposition \ref{rpfws1}, we complete the proof under this case.
    \item Suppose $\underset{n\to +\infty}{\varliminf}|a_n|>0$. Then, we have 
    $$\lim\limits_{n\to +\infty}\frac{\log(|a_n|)}{-\ell_n}=0.$$ 
    By \eqref{bad up}, we obtain $\dim_F(\widehat{\mu}_{2,a,\ell})=0$. 
    \item Suppose $\underset{n\to +\infty}{\varliminf}|a_n|=0$ but $\underset{n\to +\infty}{\varlimsup}|a_n|>0$. Since $\underset{n\to +\infty}{\varliminf}|a_n|=0$, by Lemma \ref{lbof}, there exists a sequence $\{t_k\}$ such that for any $k$ large enough,
    \[|\widehat{\mu}_{2,a,\ell}(2^{-\ell_{t_k}})|\gtrsim |a_{t_k}|.\]
    If $\underset{n\to +\infty}{\varlimsup}|a_{t_k}|>0$, then 
    \[\underset{|\xi|_p\to +\infty}{\varlimsup}|\widehat{\mu}_{2,a,\ell}(\xi)|>0,\]
    which implies $\dim_F(\widehat{\mu}_{2,a,\ell})=0$. Otherwise, we have $\underset{n\to +\infty}{\lim}|a_{t_k}|=0$. By the same argument in the proof of case (i), we have that for any $k$ large enough,
    \[|\widehat{\mu}_{2,a,\ell}(2^{-\ell_{t_k+1}})|\gtrsim |a_{t_k+1}|.\]
    Since $\underset{n\to +\infty}{\varlimsup}|a_n|>0$, iterating the preceding argument yields a sequence of indices $\{s_k\}$ with $t_{k}\le s_k\le t_{k+1}$ such that $\underset{k\to +\infty}{\varlimsup}|a_{s_k}|>0$. Hence, $\dim_F(\widehat{\mu}_{2,a,\ell})=0$.
  \end{enumerate}
\end{proof}
\begin{proof}[Proof of Theorem \ref{thm1.5}]
  Combining Propositions \ref{rpfws1}, \ref{proof of lower bound} and \ref{flyws2}, we complete the proof.
\end{proof}


%
%
%
%

 { }


\vspace*{10ex}

\noindent Cheng Liu: School of Mathematics and Statistics,
Wuhan University, 
\vspace{-2mm}

\noindent\phantom{Cheng Liu: }Bayi Road 299, Wuchang District, Wuhan, China

\noindent\phantom{Cheng Liu: }e-mail: 2023202010021@whu.edu.cn 


\vspace{5mm}

\noindent Lingmin Liao: School of Mathematics and Statistics,
Wuhan University, 
\vspace{-2mm}

\noindent\phantom{Lingmin Liao: }Bayi Road 299, Wuchang District, Wuhan, China

\noindent\phantom{Lingmin Liao: }e-mail: lmliao@whu.edu.cn 

\end{document}